\documentclass[12pt]{article}
\usepackage{amssymb,amsmath}
\usepackage{cases}
\usepackage{amsthm}
\usepackage{amsfonts}
\usepackage{cite,color,xcolor}
\usepackage[left=2.5cm,right=2.5cm,top=2.0cm,bottom=2.0cm]{geometry}
\usepackage[colorlinks,citecolor=blue,urlcolor=blue]{hyperref}
\usepackage{tikz}

\usetikzlibrary{patterns}

\newtheorem{theorem}{Theorem}[section]

\newtheorem{lemma}{Lemma}[section]

\newtheorem{remark}{Remark}[section]

\newcommand{\bal}{\begin{align}}
\newcommand{\bbal}{\begin{align*}}
\newcommand{\beq}{\begin{equation}}
\newcommand{\eeq}{\end{equation}}
\newcommand{\bca}{\begin{cases}}
\newcommand{\eca}{\end{cases}}

\newcommand{\pa}{\partial}
\newcommand{\fr}{\frac}

\newcommand{\de}{\delta}

\newcommand{\dd}{\mathrm{d}}

\newcommand{\R}{\mathbb{R}}

\newcommand{\les}{\lesssim}

\newcommand\f{\left}
\newcommand\g{\right}
\begin{document}
\title{Norm inflation and ill-posedness for the Fornberg-Whitham equation}

\author{
Jinlu Li
\footnote{School of Mathematics and Computer Sciences, Gannan Normal University, Ganzhou 341000, China
(E-mail: lijinlu@gnnu.edu.cn(J.L))},
Xing Wu
\footnote{College of Information and Management Science, Henan Agricultural University, Zhengzhou 450002, China
(E-mail: ny2008wx@163.com(X.W))},
Yanghai Yu
\footnote{School of Mathematics and Statistics, Anhui Normal University, Wuhu 241002, China
(E-mail: yuyanghai214@sina.com(Y.Y))} and
Weipeng Zhu\footnote{School of Mathematics and Big Data, Foshan University, Foshan, Guangdong 528000, China
(E-mail: mathzwp2010@163.com(W.Z))}}

\date{\today}

\maketitle\noindent{\hrulefill}

{\bf Abstract:} In this paper, we prove that the Cauchy problem for the Fornberg-Whitham equation is not locally well-posed in $B^s_{p,r}(\R)$  with $(s,p,r)\in (1,1+\frac1p)\times[2,\infty)\times [1,\infty]$ or $(s,p,r)\in \{1\}\times[2,\infty)\times [1,2]$ by showing norm inflation phenomena of the solution for some special initial data.


\vskip0mm\noindent{\hrulefill}

\section{Introduction}\label{sec1}
\subsection{Model and Known Results}\label{subsec11}
\bibliographystyle{plain}
In this paper, we focus on the Cauchy problem of the Fornberg-Whitham (FW) equation
\begin{align}\label{FW}
\begin{cases}
u_t-u_{x x t}+\frac{3}{2} u u_x-\frac{9}{2} u_x u_{x x}-\frac{3}{2} u u_{x x x}-u_x=0, &\quad (t,x)\in \R^+\times\R,\\
u(0, x)=u_0(x),&\quad x\in \R,
\end{cases}
\end{align}
which was first introduced by Whitham \cite{Whitham} in 1967 and Whitham and Fornberg \cite{FW} in 1978, as a shallow water wave model to study the qualitative behaviors of wave breaking (the solution remains bounded while its slope
becomes unbounded in finite time).

The FW equation was compared with the famous Korteweg-de Vries (KdV) equation \cite{Korteweg}
\begin{eqnarray}\label{eq4}
          u_t+6uu_x=-u_{xxx},
\end{eqnarray}
and the classical Camassa-Holm (CH) equation \cite{Camassa1993,Constantin-E,Escher1,Escher2,Escher3,Escher4,Escher5,Fokas1981}
\begin{eqnarray}\label{ch}
          u_t-u_{xxt}+3uu_x=2u_xu_{xx}+uu_{xxx}.
                  \end{eqnarray}
The KdV equation admits solitons or solitary traveling wave solutions. Indeed, the KdV equation in the non-periodic admits the solitary wave solutions with the following form
$$u(t,x)=\frac{c}{2}{\rm{sech}}^2\Big(\frac{\sqrt{c}}{2}(x-ct)\Big).$$
And, the CH equation possess exact peaked soliton solutions (peakons) of the form $$u(t,x)=ce^{-|x-ct|}.$$
It is interesting that the FW equation does not only admit solitary traveling wave solutions like the KdV
equation, but also possess peakon solutions (or peaked traveling wave solutions) \cite{Whitham} as the CH equation which
are of the form
\bal\label{pek}
u(t,x)=\frac{8}{9}e^{-\frac{1}{2}|x-\frac{4}{3}t|}.\end{align}
The $\mathrm{FW}$ equation \eqref{FW} admits the conserved quantities \cite{Naumkin}
\bbal
&E_1(u)=\int_{\mathbb{R}} u \dd x,\quad
E_2(u)=\int_{\mathbb{R}} u^2 \dd x, \\
& E_3(u)=\int_{\mathbb{R}}\left(u\left(1-\partial_x^2\right)^{-1} u-u^3\right) \dd x.
\end{align*}

A classification of other traveling wave solutions of the FW equation was presented by Yin, Tian and Fan \cite{Yin}. It's worth mentioning that the KdV equation and CH equation are integrable,
and they possess infinitely many conserved quantities, an infinite hierarchy of quasi-local symmetries, a
Lax pair and a bi-Hamiltonian structure.  Unlike the Kdv and CH equation, the FW equation is not integrable. Although the FW equation is in a simple form, the only useful
conservation law we know so far is $\|u\|_{L^2}$. Therefore, the analysis of the FW equation would be somewhat
more difficult due to the special structure of this equation and the lower regularity of the conservation law. Particularly, the well-posedness theory for the FW equation is not completely understood.
 Before recalling the well-posedness results the FW equation, we firstly transform the FW equation \eqref{FW} equivalently into the following non-local form
\begin{align}\label{b}
\begin{cases}
\pa_tu+\frac32uu_x=\pa_x(1-\pa^2_x)^{-1}u, &\quad (t,x)\in \R^+\times\R,\\
u(0,x)=u_0(x), &\quad x\in \R.
\end{cases}
\end{align}

By the Galerkin approximation
argument, Holmes \cite{Holmes16} proved the well-posed of FW equation in Sobolev spaces $H^s(\mathbb{T})$ with $s > 3/2$. Holmes and Thompson \cite{Holmes17} obtained the well-posedness of FW equation in Besov spaces $B^s_{2,r}(\mathbb{R}\;\text{or}\; \mathbb{T})(s>3/2,1<r<\infty\;\text{or}\;s=3/2,r=1)$. They also proved that the data-to-solution map
is not uniformly continuous but H\"{o}lder continuous in some given topology and presented a blow-up criterion for solutions. Later, Haziot\cite{Haziot}, H\"{o}rmann\cite{G18-1,G3}, Wei\cite{Wei18,Wei21}, Wu-Zhang\cite{Wu} and Yang \cite{Yang} sharpened this
blowup criterion and presented the sufficient conditions about the initial data to lead the wave-breaking phenomena of
the FW equation. The discontinuous traveling waves as weak solutions to the FW equation were
investigated in \cite{G18-2}.
Recently, Guo\cite{Guo} established the local well-posedness (existence, uniqueness and continuous dependence) for the FW equation in both supercritical Besov spaces $B_{p, r}^s$ with $s>1+\frac{1}{p}$, $(p, r)\in [1,\infty]\times[1,\infty)$ and critical Besov spaces $B_{p, 1}^{1+1/p}$ with $p\in[1,\infty)$. We also should mention that the FW equation is ill-posed in $B_{2, \infty}^{3/2}$ by the peakon solution \eqref{pek} and in $B^s_{p,\infty}$ $(s>1+1/{p}, p\in[1,\infty])$ by proving the solution map starting from $u_0$ is discontinuous at $t = 0$ in the metric of $B^s_{p,\infty}$ (see our recent papers \cite{Li22,Li22-2}). For the CH equation, Guo-Liu-Molinet-Yin \cite{Guo-Yin} showed the ill-posedness in the critical Besov spaces $B_{p,r}^{1+1/p}(\mathbb{R}\;\text{or}\; \mathbb{T})$ with $(p,r)\in[1,\infty]\times(1,\infty]$ (especially in $H^{3/2}$) due to norm inflation. Following their methods, the FW equation is also ill-posed in the above critical Besov spaces.
However, the well-posedness issue of the FW equation in Besov spaces $B_{p, r}^s$ with $s<1+\frac{1}{p}$ appears unsolved.
In this paper, our aim is to provide a rigorous proof of ill-posedness for the FW equation in the low regularity setting by exhibiting a strong instability known as a {\it Norm Inflation}, i.e. arbitrary small data leads to an arbitrarily large
solution in arbitrary small time, preventing the solution map to be continuous.

\subsection{Main Result}\label{subsec12}
Now let us state our main ill-posedness result of this paper.
\begin{theorem}\label{th}
Assume that
\bal\label{condition}
1<s<1+\frac1p,(p,r)\in [2,\infty)\times [1,\infty]  \quad \mathrm{or} \quad s=1, (p,r)\in[2,\infty)\times[1,2].
\end{align}
For any $\de>0$, there exists initial data satisfying
\bbal
\|u_0\|_{B^s_{p,r}}\leq \de,
\end{align*}
such that a solution $u(t)\in \mathcal{C}([0,T_0];B^s_{p,r})$ of the Cauchy problem \eqref{b} satisfies
\bbal
\|u(T_0)\|_{B^s_{p,r}}\geq \frac1\de \quad \mathrm{for} \ \mathrm{some} \quad 0<T_0<\delta.
\end{align*}
\end{theorem}
\begin{remark}\label{re1}
Theorem \ref{th} indicates the ill-posedness of \eqref{b} in $B^s_{p,r}(\R)$ in the sense that the solution map $B^s_{p,r}\ni u_0 \mapsto u\in B^s_{p,r}$ is discontinuous with respect to the initial data.
\end{remark}

\begin{remark}\label{re2}
To the best of our knowledge, Theorem \ref{th} is new for the FW equation even in the Sobolev spaces $H^s(\R)=B^{s}_{2,2}(\R)$ with $s\in[1,\fr32)$.
\end{remark}
\subsection{Strategy to Proof}\label{subsec13}
We shall outline the main ideas in the proof of Theorem \ref{th}.
\begin{itemize}
  \item Firstly, we construct an explicit example for initial data $u_0$, where the norm $\|u_0\|_{B^s_{p,r}}$ is sufficiently small.
  \item Secondly, we present the lower and upper bounds of lifespan $T^*$ of the solution to the FW equation \eqref{b} by exploring fully the properties of the flow map $\psi$.
  \item Lastly, we mainly observe that the transport term does cause growth of the $L^2$-norm of $u_x$. Precisely speaking, we estimate the $L^2$-norm of $u_x$ over $(-\psi(T_{\rm{min}},q_0),\psi(T_{\rm{min}},q_0))$ and obtain that its lower bound can be arbitrarily large.
\end{itemize}
{\bf The structure of the paper.}\quad
In Section \ref{sec2} we provide several key Lemmas. In Section \ref{sec3} we present the proof of Theorem \ref{th}.
\section{Preliminary}\label{sec2}
\subsection{Notations and Norms}\label{subsec21}
For the reader's convenience, we collect the notations and functional spaces that will be used frequently throughout this paper.
\begin{itemize}
  \item We use the letter $C$ to denote various absolute constant whose
value can change from a line to another or even within a single line. $C(s,t)$ stands for some positive constant dependent of the variables $s$ and $t$.
  \item The symbol $A\approx B$ means that $C^{-1}B\leq A\leq CB$.
  \item Given a Banach space $X$, we denote its norm by $\|\cdot\|_{X}$.
  \item We denote by $W^{1,p}$ the standard Sobolev space on $\R$ of $L^p$ functions whose derivative also belongs to $L^p$.
\item We will also define the Lipschitz space $\mathrm{Lip}$ using the norm $\|f\|_{\mathrm{Lip}}=\|f\|_{L^\infty}+\|\pa_xf\|_{L^\infty}$.
\item For $I\subset\R$, we denote by $\mathcal{C}(I;X)$ the set of continuous functions on $I$ with values in $X$. Sometimes we will denote $L^p(0,T;X)$ by $L_T^p(X)$.
    \item Let $s\in\mathbb{R}$ and $(p,r)\in[1, \infty]^2$. The nonhomogeneous Besov space $B^{s}_{p,r}(\R)$ is defined by (see \cite{B} for more details).
\begin{align*}
B^{s}_{p,r}(\R):=\f\{f\in \mathcal{S}'(\R):\;\|f\|_{B^{s}_{p,r}(\mathbb{R})}:=\f\|2^{j s}\| \Delta_j u\|_{L^p}\g\|_{\ell^r(j\geq-1)}<\infty\g\}.
\end{align*}
\end{itemize}

\subsection{Key Example for Initial Data}
Firstly, we construct an explicit example as follows.

Let $p_0\gg1$ and $0<q_0\ll1$ which will be fixed later. Set
\bal\label{ini}
u_0(x)=p_0\big(e^{-|x+q_0|}-e^{-|x-q_0|}\big).
\end{align}
That is, the initial profile $u_0(x)$ is the anti-symmetric peakon-antipeakon (see Fig.1).
\vskip0.1in
\hskip1in
\begin{minipage}{0.7\linewidth}
\hspace*{0cm}
\vspace*{0cm}
\begin{tikzpicture}[xscale=1,yscale=1]
%
%
\newcommand\X{7};
\newcommand\Y{2};
\newcommand\FX{11};
\newcommand\FY{11};
\newcommand\Z{0.6};
\newcommand\C{2};
\newcommand\A{1};
%
%
\draw[->,line width=1pt,black] (-5,0)--(5,0)
node[above left] {\fontsize{\FX}{\FY}$x$};
\draw[->,line width=1pt,black] (0,-2.5)--(0,2.5) node[below right] {\fontsize{\FX}{\FY}$u_0$};
\draw[domain=-4:4, variable=\x,
red, line width=1.5pt]
plot ({\x},{\C*exp(-abs(\x+\A))-\C*exp(-abs(\x-\A))});
\draw[line width=1pt,black,dashed]
({-(\A)},{\C-\C*exp(-abs(2*\A)))})
node[] { }
node[above,xshift=-1.0cm] {\fontsize{\FX}{\FY}
$-u_0(q_0)=u_0(-q_0)$}
--
({-(\A)},0)
node[below, xshift=-.2cm] {\fontsize{\FX}{\FY}$-q_0$}
node[] {$\bullet$} ;
\draw[line width=1pt,black,dashed]
({(\A)},{\C*exp(-abs(2*\A))-\C})
node[] {}
node[below,xshift=1.2cm] {\fontsize{\FX}{\FY}
$u_0(q_0)=p_0(e^{-2q_0}-1)$}
--
({(\A)},0)
node[above, xshift=.2cm] {\fontsize{\FX}{\FY}$q_0$} node[] {$\bullet$};
\end{tikzpicture}
\end{minipage}
\vskip0.15in
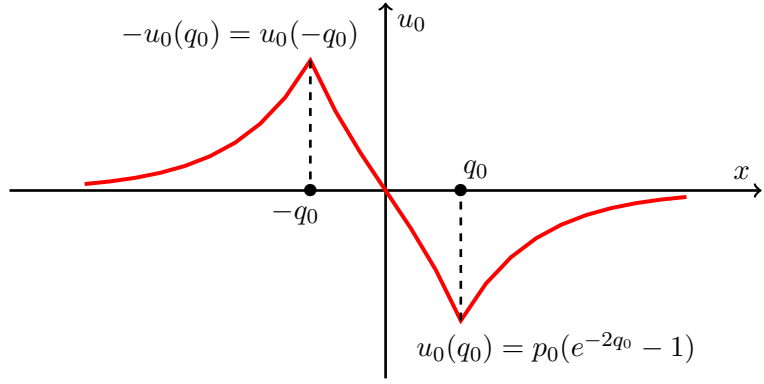
\centerline{Fig.1: Initial profile $u_0(x)$.}
\vskip0.1in
By the simple fact which will be also used in the sequel $$\fr{x}{x+1}\leq 1-e^{-x}\leq x,\quad \forall x>-1,$$ one has
$$\|u_0\|_{L^\infty}\leq 2p_0q_0\quad\text{and}\quad \|u_0\|_{L^2}\leq 4p_0q_0.$$
It is easy to check that $u_0(x)$ is an odd function and
\bal\label{ini-2}
u'_0(x)=
\begin{cases}
p_0(e^{q_0}-e^{-q_0})e^{x}, &\;\text{if}\; x\in(-\infty,-q_0),\\
-p_0e^{-q_0}(e^x+e^{-x}), &\;\text{if}\; x\in(-q_0,q_0),\\
p_0(e^{q_0}-e^{-q_0})e^{-x}, &\;\text{if}\; x\in(q_0,+\infty),
\end{cases}
\end{align}
which is displayed in Figure 2.
\vskip0.1in
\hskip1in
\begin{minipage}{0.7\linewidth}
\hspace*{0cm}
\vspace*{0cm}
\begin{tikzpicture}[xscale=1,yscale=1]
%
%
\newcommand\X{7};
\newcommand\Y{2};
\newcommand\FX{11};
\newcommand\FY{11};
\newcommand\Z{0.6};
\newcommand\C{2};
\newcommand\A{1};
%
%
\draw[->,line width=1pt,black] (-5,0)--(5,0)
node[above left] {\fontsize{\FX}{\FY}$x$};
\draw[->,line width=1pt,black] (0,-2.5)--(0,2.5) node[below left] {\fontsize{\FX}{\FY}$u'_0$};
\draw[domain=-4:-1, variable=\x,
red, line width=1.5pt]
plot ({\x},{\C*exp(\x+\A)-\C*exp(\x-\A)});
\draw[domain=-1:1, variable=\x,
red, line width=1.5pt]
plot ({\x},{-\C*exp(\x-\A)-\C*exp(-\x-\A)});
\draw[domain=1:4, variable=\x,
red, line width=1.5pt]
plot ({\x},{\C*exp(-\x+\A)-\C*exp(-\x-\A)});
\draw[domain=-1:1, variable=\x,
black,dashed, line width=1.5pt]
plot ({\x},{-\C*exp(-2*\A)-\C});
\draw[domain=-1:1, variable=\x,
black,dashed, line width=1.5pt]
plot ({\x},{\C-\C*exp(-2*\A)});
\draw[line width=1pt,black,dashed]
({(\A)},{-\C*exp(-2*\A)-\C})
node[] {}
node[below,xshift=1.2cm] {\fontsize{\FX}{\FY}
$u'_0(q_0^-)=-p_0(e^{-2q_0}+1)$}
--
({(\A)},0)
node[below, xshift=.4cm] {\fontsize{\FX}{\FY}$q_0$} node[] {$\bullet$};
\draw[line width=1pt,black,dashed]
({-(\A)},{\C-\C*exp(-2*\A)})
node[] { }
node[above,xshift=-1.0cm] {\fontsize{\FX}{\FY}}
--
({-(\A)},0)
node[below, xshift=-.4cm] {\fontsize{\FX}{\FY}$-q_0$}
node[] {$\bullet$} ;
\draw[line width=1pt,black,dashed]
({-(\A)},{-\C*exp(-2*\A)-\C})
node[] { }
node[above,xshift=-1.0cm] {\fontsize{\FX}{\FY}}
--
({-(\A)},0)
node[below, xshift=-.4cm] {\fontsize{\FX}{\FY}$-q_0$}
node[] {$\bullet$} ;
\draw[line width=1pt,black,dashed]
({(\A)},{\C-\C*exp(-2*\A)})
node[] {}
node[above,xshift=1.2cm] {\fontsize{\FX}{\FY}
$u_0(q_0^+)=p_0(1-e^{-2q_0})$}
--
({(\A)},0)
node[below, xshift=.4cm] {\fontsize{\FX}{\FY}$q_0$} node[] {$\bullet$};
\end{tikzpicture}
\end{minipage}
\vskip0.15in
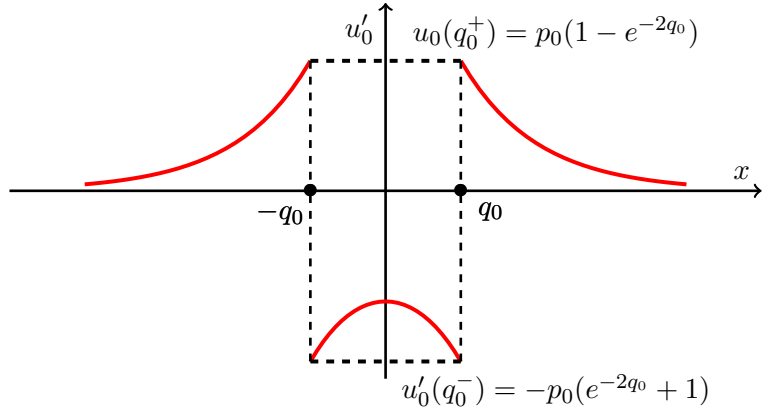
\centerline{Fig.2: Graph of $u'_0(x)$.}
\vskip0.1in

Furthermore, we can deduce that the following result holds:
\begin{lemma}\label{le01} For every $q_0\in(0,1)$ and $s\in(\fr12,\fr32)$, there exists $C=C_s>0$ such that
\begin{align*}
C^{-1} p_0q_0^{3/2-s}\leq\|u_0\|_{H^s}\leq C p_0q_0^{3/2-s}.
\end{align*}
\end{lemma}
\begin{proof} The proof essentially follows that of Lemma 3.1 in \cite{Byers} or Proposition 1 in \cite{Himonas}. \end{proof}

\subsection{Blow-up Criterion}
\begin{lemma}\label{le2} Assume that $(s,p,r)$ satisfies \eqref{condition}, there exists a maximal time $T^*$ and a solution $u\in \mathcal{C}([0,T^*);B^s_{p,r})\cap L^\infty([0,T^*);\mathrm{Lip})$ to the FW equation \eqref{b} with initial data $u_0$ given by \eqref{ini}. Moreover, if the maximal time $T^*$ is finite, we have
\bbal
\lim_{t\uparrow T^*}\big(\|u(t)\|_{B^s_{p,r}}+\|u(t)\|_{\mathrm{Lip}}\big)=+\infty\quad\Leftrightarrow\quad\lim_{t\uparrow T^*}\|u(t)\|_{\mathrm{Lip}}
=+\infty.
\end{align*}
\end{lemma}
\begin{proof}
It follows from \eqref{ini}-\eqref{ini-2} and Lemma \ref{le01} that $u_0\in H^\sigma\cap \mathrm{Lip}\subseteq B^s_{p,r}\cap \mathrm{Lip}$ with $\sigma\in (s-\frac1p+\frac12,\fr32)$.

Following the proof of Lemma 2.4 in \cite{d1}, we can get
\begin{align*}
&\|u(t)\|_{B^s_{p,r}} \leq\left\|u_{0}\right\|_{B^s_{p,r}} \exp\left(C \int_{0}^{t}1+\|u(\tau)\|_{\mathrm{Lip}} \dd \tau\right)
\end{align*}
and
\begin{align*}
&\|u(t)\|_{\mathrm{Lip}} \leq\left\|u_{0}\right\|_{\operatorname{Lip}} \exp\left(C \int_{0}^{t}1+\|u(\tau)\|_{\mathrm{Lip}} \dd \tau\right).
\end{align*}
This is enough to complete the proof of Lemma \ref{le2}.\end{proof}
\subsection{The Equation Along the Flow}
Given a Lipschitz velocity field $u$, we may solve the following ODE to find the flow induced by $u$:
\begin{align}\label{ode}
\quad\begin{cases}
\frac{\dd}{\dd t}\psi(t,x)=\frac32u(t,\psi(t,x)),\\
\psi(0,x)=x,
\end{cases}
\end{align}
which is equivalent to the integral form
\bal\label{psi}
\psi(t,x)=x+\frac32\int^t_0u(\tau,\psi(\tau,x))\dd \tau.
\end{align}
Because the velocity field is Lipschitz, then we get that for $t\in[0,T^*)$
\bbal
\pa_x\psi(t,x)=\exp\left(\frac32\int^t_0\pa_x u(\tau,\psi(\tau,x))\dd \tau\right)>0.
\end{align*}
This shows that $\psi(t,\cdot)$ is an increasing diffeomorphism over $\R$, that is,  for all $x,y\in \R,$ there holds that $\psi(t,x)<\psi(t,y)$ if $x< y$.
Also, we have
\bal\label{psi-xt}
\|\pa_x\psi(t,x)\|_{L^\infty}\leq e^{CT^*\|\pa_x u\|_{L^\infty_t(L^\infty)}} \quad\text{and}\quad \|\pa_t\psi(t,x)\|_{L^\infty}\leq C\|u\|_{L^\infty_{T^*}(L^\infty)}.
\end{align}
In view of \eqref{b}, we get that
\bbal
\frac{\dd}{\dd t}u(t,\psi(t,x))&=u_{t}(t,\psi(t,x))+u_{x}(t,\psi(t,x))\frac{\dd}{\dd t}\psi(t,x)
\\&=\left(u_t+\frac{3}{2} u u_x\right)(t, \psi(t, x))\\&=\pa_x(1-\pa^2_x)^{-1}u(t,\psi(t,x)).
\end{align*}
Integrating the above with respect to time variable yields that
\bal\label{v}
u(t,\psi(t,x))=u_0(x)+\int^t_0\pa_x(1-\pa^2_x)^{-1}u(\tau,\psi(\tau,x))\dd \tau.
\end{align}
Differentiating \eqref{b} with respect to space variable $x$, we find
\bal\label{u1}
u_{tx}+\frac32uu_{xx}+\frac32(u_x)^2=\pa^2_x(1-\pa^2_x)^{-1}u=:V.
\end{align}
Combining \eqref{ode} and \eqref{u1}, we obtain
\bal\label{du1}
\frac{\dd}{\dd t}u_x(t,\psi(t,x))&=u_{tx}(t,\psi(t,x))+u_{xx}(t,\psi(t,x))\frac{\dd}{\dd t}\psi(t,x)\nonumber\\
&=\left(u_{tx}+\frac32uu_{xx}\right)(t,\psi(t,x))\nonumber\\
&=-\frac32(u_x)^2(t,\psi(t,x))+V(t,\psi(t,x)),
\end{align}
which means that
\bal\label{ux1}
u_x(t,\psi(t,x))=u'_0(x)-\fr32\int^t_0(u_x)^2(\tau,\psi(\tau,x))\dd \tau+\int^t_0V(\tau,\psi(\tau,x))\dd \tau.
\end{align}
Differentiating \eqref{u1} with respect to space variable $x$, we find
\bal\label{u2}
u_{txx}+\frac32uu_{xxx}+\frac92u_xu_{xx}=\pa^3_x(1-\pa^2_x)^{-1}u,
\end{align}
which gives us that
\bal\label{du2}
&\frac{\dd}{\dd t}u_{xx}(t,\psi(t,x))
=-\frac92u_xu_{xx}(t,\psi(t,x))+\pa^3_x(1-\pa^2_x)^{-1}u(t,\psi(t,x))
\end{align}
and
\bal\label{du3}
u_{xx}(t,\psi(t,x))=u''_0(x)-\fr92\int^t_0u_xu_{xx}(\tau,\psi(\tau,x))\dd \tau+\int^t_0\pa^3_x(1-\pa^2_x)^{-1}u(\tau,\psi(\tau,x))\dd \tau.
\end{align}
\subsection{Estimation of Lifespan}
We need two preliminary Lemmas.
\begin{lemma}\label{le0} Assume that $(s,p,r)$ satisfies \eqref{condition}. Let $u\in \mathcal{C}([0,T^*);B^s_{p,r})\cap L^\infty([0,T^*);\mathrm{Lip})$ be the solution of the FW equation \eqref{b} with initial data $u_0$ given by \eqref{ini}, then we have
\bbal
u_{x}(t,\psi(t,x))\in \mathcal{C}^1([0,T^*)\times(-q_0,q_0)).
\end{align*}
\end{lemma}
\begin{proof} From \eqref{psi-xt}, it is not difficult to check that for any $\tau\in[0,t]$ and $x,y\in (-q_0,q_0)$
\bbal
|V(\tau,\psi(\tau,x))-V(\tau,\psi(\tau,y))|&\leq C\|\pa_xu\|_{L^\infty_t(L^\infty)}|\psi(\tau,x)-\psi(\tau,y)|
\\&\leq C|x-y|\cdot\|\pa_xu\|_{L^\infty_t(L^\infty)} e^{CT^*\|\pa_x u\|_{L^\infty_t(L^\infty)}}
\end{align*}
and
\bbal
|V(t,\psi(t,x))-V(s,\psi(s,x))|&\leq C|t-s|\cdot(\|\pa_tu\|_{L^\infty_t(L^\infty)}+\|\pa_xu\|_{L^\infty_t(L^\infty)}\|\pa_t\psi\|_{L^\infty_t(L^\infty)})
\\&\leq C(T^*,\|u\|_{L^\infty_{t}(\mathrm{Lip})}) |t-s|,
\end{align*}
which mean that $V(t,\psi(t,x)) \in \mathcal{C}([0,T^*)\times(-q_0,q_0))$.

Due to \eqref{ux1}, one has for any $x,y\in (-q_0,q_0)$
\bbal
X(t)&:=|u_x(t,\psi(t,x))-u_x(t,\psi(t,y))|\\
&\leq\left|u'_0(x)-u'_0(y)\right|+\int^t_0\left|V(\tau,\psi(\tau,x))-V(\tau,\psi(\tau,y))\right|\dd \tau\\
&\quad+\frac32\int^t_0\left|X(\tau)(u_x(\tau,\psi(\tau,x))+u_x(\tau,\psi(\tau,y))\right|\dd \tau\\
&\leq\left|u'_0(x)-u'_0(y)\right|+CT^*|x-y|\cdot \|\pa_xu\|_{L^\infty_t(L^\infty)} e^{CT^*\|\pa_x u\|_{L^\infty_t(L^\infty)}}\\
&\quad+3\int^t_0\left|X(\tau)\right|\|u_x\|_{L^\infty}\dd \tau,
\end{align*}
from which, using Gronwall's inequality, we have
\bal\label{lwyz}
|u_x(t,\psi(t,x))-u_x(t,\psi(t,y))|
&\leq C_t\left(\left|u'_0(x)-u'_0(y)\right|+|x-y|\right),
\end{align}
where $C_t$ depends on $T^*$ and $\|\pa_x u\|_{L^\infty_{t}(L^\infty)}$.

Similarly, one has for any $s,t\in[0,T^*)$ and $x\in (-q_0,q_0)$
\bbal
|u_x(t,\psi(t,x))-u_x(s,\psi(s,x))|\leq C\f(\|u\|_{L^\infty_{\max\{t,s\}}(L^\infty)}+\|\pa_x u\|^2_{L^\infty_{\max\{t,s\}}(L^\infty)}\g)|t-s|.
\end{align*}
Noticing that $u'_0(x)\in \mathcal{C}(-q_0,q_0)$, we deduce
\bbal
u_x(t,\psi(t,x))\in \mathcal{C}([0,T^*)\times(-q_0,q_0)).
\end{align*}
It is easy from \eqref{du3} to obtain that $u_{xx}(\tau,x)\in L^\infty_{t}L^\infty$ for $\tau\in[0,t]$.
Repeating the above procedure, we get
\bbal
u_{xx}(t,\psi(t,x))\in \mathcal{C}([0,T^*)\times(-q_0,q_0)),
\end{align*}
which in turn follows from \eqref{du1} that
\bbal
\frac{\dd}{\dd t}u_{x}(t,\psi(t,x))\in \mathcal{C}([0,T^*)\times(-q_0,q_0)).
\end{align*}
This completes the proof of Lemma \ref{le0}.
\end{proof}

We should emphasize that $u_x(t,x)$ is discontinuous in $[0,T^*)\times \R$, but we can claim that $u_x(t,x)$ is continuous in $[0,T^*)\times (\psi(t,-q_0),\psi(t,q_0))$. Furthermore, we have

\begin{lemma}\label{le1} Assume that $(s,p,r)$ satisfies \eqref{condition}. Let $u\in \mathcal{C}([0,T^*);B^s_{p,r})\cap L^\infty([0,T^*);\mathrm{Lip})$ be the solution of the FW equation \eqref{b} with initial data $u_0$ given by \eqref{ini}, then we have
\bbal
\pa_tu_x(t,x)\in \mathcal{C}([0,T^*)\times(\psi(t,-q_0),\psi(t,q_0))).
\end{align*}
\end{lemma}
\begin{proof}
Let $z\in (\psi(t,-q_0),\psi(t,q_0))$. Due to \eqref{psi}, we have
\bbal
\psi^{-1}(t,z)=z-\frac32\int^t_0u(\tau,\psi(\tau,\psi^{-1}(t,z)))\dd \tau.
\end{align*}
Also, we have for $t\in[0,T^*)$
\bal\label{lwyz1}
&\|\pa_z\psi^{-1}(t,z)\|_{L^\infty}\leq C(T^*,\|\pa_x u\|_{L^\infty_{t}(L^\infty)}),\\
&\|\pa_t\psi^{-1}(t,z)\|_{L^\infty}\leq C(T^*,\|u\|_{L^\infty_{t}(\mathrm{Lip})}). \nonumber
\end{align}
From which and Lemma \ref{le0}, we obtain that  for $s,t\in[0,T^*)$ and $x,y\in (\psi(t,-q_0),\psi(t,q_0))$
\bbal
|u_x(t,x)-u_x(s,y)|&\leq|u_x(t,x)-u_x(t,y)|+|u_x(t,y)-u_x(s,y)|\\
&\leq|u_x(t,\psi(t,\psi^{-1}(t,x)))-u_x(t,\psi(t,\psi^{-1}(t,y)))|\\
&~~~~+|u_x(t,\psi(t,\psi^{-1}(t,y)))-u_x(s,\psi(s,\psi^{-1}(s,y)))|\\
&\to0\quad\text{as}\quad (t,x)\to (s,y),
\end{align*}
which tells us that
\bbal
u_{x}(t,x)\in \mathcal{C}([0,T^*)\times(\psi(t,-q_0),\psi(t,q_0))).
\end{align*}
It is easy from \eqref{du2} and \eqref{du3} to obtain that $u_{xxx}(\tau,x),u_{xxt}(\tau,x)\in L^\infty_{t}L^\infty$ for $\tau\in[0,t]$.
Repeating the above procedure, we get
\bbal
u_{xx}(t,x)\in \mathcal{C}([0,T^*)\times(\psi(t,-q_0),\psi(t,q_0))).
\end{align*}
Notice that $\pa_tu_x=-\frac32(u_x)^2-\frac32uu_{xx}+\pa^2_x(1-\pa^2_x)^{-1}u$, we can deduce that
\bbal
\pa_tu_x(t,x)\in \mathcal{C}([0,T^*)\times(\psi(t,-q_0),\psi(t,q_0))).
\end{align*}
This completes the proof of Lemma \ref{le1}.
\end{proof}

Next, we present the lower and upper bounds of lifespan $T^*$ of the solution to the FW equation \eqref{b}.

The $L^{\infty}$-norm of any function $u$ is preserved under the flow $\psi$, i.e.
$$
\|u(t, x)\|_{L^{\infty}(\R)}=\|u(t, \psi(t, x))\|_{L^{\infty}(\R)}.
$$
Combing the above and \eqref{v}, we obtain that for $t\in[0,T^*)\cap[0,\ln3]$
$$\|u(t)\|_{L^\infty(\R)}\leq \|u_0\|_{L^\infty(\R)}+\int^t_0\|u(\tau)\|_{L^\infty(\R)}\dd \tau.$$
Using Gronwall's inequality gives us that for $t\in[0,T^*)\cap[0,\ln3]$
$$\|u(t)\|_{L^\infty(\R)}\leq \|u_0\|_{L^\infty(\R)} e^{t}\leq 3p_0q_0.$$
Note that $\left(1-\partial_x^2\right)^{-1} f=G * f$ with $G(x)= \frac{1}{2} e^{-|x|}$, by Young's inequality, we have
$$
\|\partial_x(1-\partial_x^2)^{-1} u\|_{L^{\infty}}=\left\|G_x * u\right\|_{L^{\infty}} \leq\left\|G_x\right\|_{L^1}\|u\|_{L^\infty}\leq\|u\|_{L^\infty}.
$$
In view of the identity $\pa^2_x(1-\pa^2_x)^{-1}u=-u+(1-\pa^2_x)^{-1}u$, we have
\bbal
|V(t,\psi(t,x))|\leq \|V(t,x)\|_{L^{\infty}}&\leq \|u\|_{L^\infty}+\|(1-\partial_x^2)^{-1} u\|_{L^\infty}\leq 2\|u\|_{L^\infty}\leq 6p_0 q_0.
\end{align*}
Then for all $t\in[0,T^*)\cap[0,\ln3]$ and $x\in\R$, we deduce from \eqref{du1} that
\begin{align}\label{t1}
\begin{cases}
\frac{\dd}{\dd t}u_x(t,\psi(t,x))\leq -\frac32(u_x)^2(t,\psi(t,x))+6p_0 q_0,\\
\frac{\dd}{\dd t}u_x(t,\psi(t,x))\geq -\frac32(u_x)^2(t,\psi(t,x))-6p_0 q_0.
\end{cases}
\end{align}
From Fig 2, we see that $u'_0(x)\leq u'_0(0)=-2p_0e^{-q_0}< -p_0$ if $x\in(-q_0,q_0)$. Thus the standard argument of continuity shows
\bbal
u_x(t,\psi(t,x))\leq -p_0, \quad \forall (t,x)\in[0,T^*)\cap[0,\ln3]\times (-q_0,q_0),
\end{align*}
from which and \eqref{t1}, we have for $(t,x)\in [0,T^*)\cap[0,\ln3]\times (-q_0,q_0)$
\begin{align}\label{t2}
\begin{cases}
\frac{\dd}{\dd t}u_x(t,\psi(t,x))\leq -\frac32\left(u_x(t,\psi(t,x))+p_0 q_0\right)^2,\\
\frac{\dd}{\dd t}u_x(t,\psi(t,x))\geq -\frac32\left(u_x(t,\psi(t,x))-p_0 q_0\right)^2.
\end{cases}
\end{align}
Solving the differential inequalities \eqref{t2} yields for $(t,x)\in [0,T^*)\cap[0,\ln3]\times (-q_0,q_0)$
\bal\label{zw}
\frac{1}{\frac32t+\frac{1}{u'_0(x)-p_0 q_0}}+p_0 q_0\leq u_x(t,\psi(t,x))\leq \frac{1}{\frac32t+\frac{1}{u'_0(x)+p_0 q_0}}-p_0 q_0.
\end{align}
According to the definition of $u_0$, we can deduce that the lifespan $T^*$ of the solution to \eqref{b} satisfies
\bbal
T^* &\leq \frac23\inf_{x\in(-q_0,q_0)}\left\{\frac{-1}{u'_0(x)+p_0 q_0}\right\}\\
& =-\frac{2}{3}\frac{1}{u'_0(q^-_0)+p_0 q_0}\\
&=\frac{2}{3}\frac{1}{p_0(1+e^{-2q_0}- q_0)}=:T_{\mathrm{max}},
\end{align*}
where we have used that $u'_0(q_0^-)=-p_0(e^{-2q_0}+1)$.

It is easy to deduce  that for sufficiently small $q_0$ $$T_{\mathrm{max}}\in \left(\fr1{3p_0},\fr2{3p_0}\right).$$

  Now, we claim that
 \bal\label{cla}
 |u_x(t,\psi(t,x))|\leq 6p_0 q_0,\quad \forall (t,x)\in [0,T^*)\times (-\infty, -q_0)\cap(q_0,+\infty),
 \end{align}
 which tells us that the solution do not blow up when $(t,x)\in [0,T^*)\times (-\infty, -q_0)\cap(q_0,+\infty)$.

Next, we begin to prove \eqref{cla}.
Due to $u'_0(x)\in(0,2p_0q_0)$ if $x\in(-\infty, -q_0)$,
from $\eqref{t1}_1$,
we have
$$u_x(t,\psi(t,x))\leq u'_0(x)+6p_0 q_0T^*\leq6p_0 q_0,\quad \forall (t,x)\in [0,T^*)\times (-\infty, -q_0).$$
Using the standard argument of continuity again, it follows from $\eqref{t1}_2$ that
$$u_x(t,\psi(t,x))\geq -6p_0 q_0,\quad \forall (t,x)\in [0,T^*)\times (-\infty, -q_0).$$
In summary, it holds that
$$|u_x(t,\psi(t,x))|\leq 6p_0 q_0,\quad \forall (t,x)\in [0,T^*)\times (-\infty, -q_0).$$
Following the above argument, we also have
$$|u_x(t,\psi(t,x))|\leq 6p_0 q_0,\quad \forall (t,x)\in t,x)\in [0,T^*)\times (q_0,+\infty),$$
thus the claim holds.

Then, we can deduce that
\bbal
T^* &\geq \frac23\inf_{x\in(-q_0,q_0)}\left\{\frac{-1}{u'_0(x)-p_0 q_0}\right\}\\
&= -\frac{2}{3}\frac{1}{u'_0(q^-_0)-p_0 q_0}\\
&=\frac{2}{3}\frac{1}{p_0(1+e^{-2q_0}+q_0)}=:T_{\mathrm{min}}.
\end{align*}

\section{Proof of Theorem \ref{th}}\label{sec3}
In this section, we prove Theorem \ref{th}.

According to the definition of $u'_0(x)$ (see Fig 2), we deduce from \eqref{zw} that
\bal\label{yh}
m(t)+p_0 q_0\leq u_x(t,\psi(t,x))\leq M(t)-p_0 q_0,
\end{align}
where, for notational convenience, we denote
\bal\label{zw1}
m(t):=\frac{1}{\frac32t+\frac{1}{u'_0(q_0^-)-p_0 q_0}}\quad\text{and}\quad M(t):=\frac{1}{\frac32t+\frac{1}{u'_0(0)+p_0 q_0}}.
\end{align}

Set $w(t,x):=u_x(t,x)$, then we obtain from \eqref{b}
\bbal
\pa_tw+\frac32\pa_x(uw)=V,
\end{align*}
which implies that
\bal\label{y}
\pa_t(w^2)+\frac32\pa_x(uw^2)+\frac32\pa_xuw^2=2Vw.
\end{align}
Integrating \eqref{y} with respect to space variable $x$ over $[\psi(t,-q_0),\psi(t,q_0)]$, we have
\bal\label{y1}
\int_{\psi(t,-q_0)}^{\psi(t,q_0)}\pa_t(w^2)\dd x+\frac32\int_{\psi(t,-q_0)}^{\psi(t,q_0)}\pa_x(uw^2)\dd x+\frac32\int_{\psi(t,-q_0)}^{\psi(t,q_0)}\pa_xuw^2\dd x=2\int_{\psi(t,-q_0)}^{\psi(t,q_0)}wV\dd x.
\end{align}
Direct computations yield that
\bal\label{y2}
\int_{\psi(t,-q_0)}^{\psi(t,q_0)}\pa_t(w^2)\dd x&=\frac{\dd}{\dd t}\int_{\psi(t,-q_0)}^{\psi(t,q_0)}w^2\dd x-\frac32u(t,\psi(t,q_0))w^2(t,\psi^-(t,q_0)) \nonumber
\\&\quad +\frac32u(t,\psi(t,-q_0))w^2(t,\psi^+(t,-q_0))
\end{align}
and
\bal\label{y3}
\int_{\psi(t,-q_0)}^{\psi(t,q_0)}\pa_x(uw^2)\dd x=u(t,\psi(t,q_0))w^2(t,\psi^-(t,q_0))-u(t,\psi(t,-q_0))w^2(t,\psi^+(t,-q_0)).
\end{align}
{\bf Note}\; We should mention that the both $w(t,\psi^-(t,q_0))$ and $w(t,\psi^+(t,-q_0))$ exist since $w(t,x)$ is uniformly continuous with respect to $x$ over $(\psi(t,-q_0),\psi(t,q_0))$ by \eqref{lwyz} and \eqref{lwyz1}.

Inserting \eqref{y2} and \eqref{y3} into \eqref{y1} yields
\bal\label{y4}
\frac{\dd}{\dd t}\int_{\psi(t,-q_0)}^{\psi(t,q_0)}w^2\dd x+\frac32\int_{\psi(t,-q_0)}^{\psi(t,q_0)}\pa_xuw^2\dd x=2\int_{\psi(t,-q_0)}^{\psi(t,q_0)}wV\dd x.
\end{align}
To simplify notation let $$A(t):=\int_{\psi(t,-q_0)}^{\psi(t,q_0)}w^2(t,x)\dd x,$$
combining \eqref{yh}, then \eqref{y4} reduces to
\bbal
A'(t)&=-\frac{3}{2}\int_{\psi(t,-q_0)}^{\psi(t,q_0)}u_x(t,x)w^2\dd x+2\int_{\psi(t,-q_0)}^{\psi(t,q_0)}wV\dd x\\
&\geq -\frac{3}{2}M(t)A(t)+\frac{3}{2}p_0 q_0A(t)-2\|u_0\|_{L^2}A^{\frac12}(t)\\
&\geq \left(-\frac{3}{2}M(t)-1\right)A(t)-C p_0^2 q_0^2,
\end{align*}
where we have used that
$$\|V\|_{L^2} \leq\|u\|_{L^2} =\|u_0\|_{L^2} \leq 4 p_0 q_0.$$
Solving the above differential inequality gives us that
\bal\label{A}
A(t)\geq \exp(B(t))\left(A_0-Cp_0^2 q_0^2\int^t_0 \exp(-B(\tau))\dd\tau\right),
\end{align}
where $$B(t)=\int^t_0 \left(-\frac{3}{2}M(t)-1\right)\dd\tau.$$
 Easy computations give that
\bbal
&\exp(B(t))=\frac{M(t)}{M(0)}e^{-t},\\
&A_0=\int_{|x|\leq q_0}\big(u'_0(x)\big)^2\dd x\approx  p^2_0q_0,\\
&M(0)=u'_0(0)+p_0 q_0=p_0 q_0-2p_0e^{-q_0},
\end{align*}
thus we have
\bbal
M(T_{\mathrm{min}})&=\frac{1}{\frac32T_{\mathrm{min}}+\frac{1}{u'_0(0)+p_0 q_0}}
\\&=\frac{1}{\frac{1}{p_0 q_0-u'_0(q_0^-)}+\frac{1}{u'_0(0)+p_0 q_0}}\\
&=\frac{p_0 q_0-u'_0(q_0^-)}{u'_0(0)+2p_0 q_0-u'_0(q_0^-)}M(0)\\
&\approx \frac{M(0)}{q_0},
\end{align*}
which gives directly that
\bbal
\frac{M(T_{\mathrm{min}})}{M(0)}
&\approx \frac{1}{q_0}.
\end{align*}
Also,
\bbal\int^{T_{\mathrm{min}}}_0 \exp(-B(\tau))\dd\tau&=\int^{T_{\mathrm{min}}}_0 \left(\frac32\tau(u'_0(0)+p_0 q_0)+1\right)e^{\tau}\dd\tau\\
&\les p_0 q_0T_{\mathrm{min}}^2+T_{\mathrm{min}}\\
&\les \frac{1}{p_0}.
\end{align*}
From \eqref{A}, we have
\bbal
A(T_{\mathrm{min}})\geq C\left(p^2_0-p_0q_0\right) \geq C p^2_0.
\end{align*}
Notice that \eqref{yh} again, we have
\bbal
\left(\int_{\psi(T_{\mathrm{min}},-q_0)}^{\psi(T_{\mathrm{min}},q_0)}w^p(T_{\mathrm{min}},x)\dd x\right)^{\frac1p}&\geq A^{\frac1{p}}(T_{\mathrm{min}})\left(p_0 q_0-M(T_{\mathrm{min}})\right)^{1-\frac2{p}}\\
&\geq cp_0^{\frac2{p}}\left(p_0 q_0-C\frac{M(0)}{q_0}\right)^{1-\frac2{p}}\\
&\geq cp_0q_0^{\frac{2-p}{p}}\\
&\geq cp_0.
\end{align*}
By Lemma \ref{le01}, one has for $\sigma\in (s-\frac1p+\frac12,\fr32)$
\bbal
\|u_0\|_{B^s_{p,r}}\leq C\|u_0\|_{B^{s-\frac1p+\frac12}_{2,r}}\leq C\|u_0\|_{H^\sigma}\leq C_1p_0q^{\frac32-\sigma}_0\leq \delta,
\end{align*}
but for $T_0\in[0,T_{\rm{min}})$
\bbal
\|u(T_{0})\|_{B^s_{p,r}}\geq \|u(T_{0})\|_{W^{1,p}}
&\geq c_2\sqrt{p_0}\geq \frac{1}{\delta},
\end{align*}
if some large $p_0$ and small $q_0$ is chosen. This completes the proof of Theorem \ref{th}.

\section*{Acknowledgments}
J. Li is supported by the National Natural Science Foundation of China (11801090 and 12161004) and Jiangxi Provincial Natural Science Foundation (20212BAB211004). Y. Yu is supported by the National Natural Science Foundation of China (12101011) and Natural Science Foundation of Anhui Province (1908085QA05). W. Zhu is supported by the National Natural Science Foundation of China (12201118) and Guangdong
Basic and Applied Basic Research Foundation (2021A1515111018).

\section*{Data Availability} No data was used for the research described in the article.

\section*{Conflict of interest}
The authors declare that they have no conflict of interest.

\end{document}